%% file: weak_transp.tex
\documentclass[
9pt, 
a4paper, 
oneside, 
headinclude,footinclude, 
BCOR5mm, 
]{scrartcl}

\input{structure.tex} 

\title{\normalfont\spacedallcaps{Weak transport equation on a bounded domain}} 

\subtitle{\spacedallcaps{Stability theory aprés DiPerna-Lions}} 

\author{\spacedlowsmallcaps{Jacopo Tenan\textsuperscript{1}}} 

\date{} 

\newcommand{\R}{\mathbb{R}}
\newcommand{\N}{\mathbb{N}}

\begin{document}


\renewcommand{\sectionmark}[1]{\markright{\spacedlowsmallcaps{#1}}} 
\lehead{\mbox{\llap{\small\thepage\kern1em\color{halfgray} \vline}\color{halfgray}\hspace{0.5em}\rightmark\hfil}} 

\pagestyle{scrheadings} 


\maketitle 

\setcounter{tocdepth}{2} 


\section*{Abstract} 
The article studies the transport equation that governes the motion of a fluid in a bounded domain, under the hypothesis of zero velocity at the boundary and supposing the incompressible nature of the fluid. Together with existence and uniqueness results, we study the DiPerna-Lions stability problem, in the case of a bounded domain.

\tableofcontents 




\let\thefootnote\relax\footnotetext{\textsuperscript{1} \textit{Dipartimento di Matematica, Università degli Studi di Roma "Tor Vergata", e-mail: \href{mailto:tenan@mat.uniroma2.it}{tenan@mat.uniroma2.it}}}




\section{Introduction and classical theory}


In the following, we will always consider $\Omega$ to be an open, bounded, connected and simply connected domain in $\R^N$, with smooth boundary. For every time $t$, it is possible to consider the density (or \textit{concentration}) of a fluid over this domain, quantified by the function $\rho(t):\Omega\to\R$. Under the action of a vector field $u=u(x,t)\in\R^N$, that measures the velocity of the motion of the fluid at the point $x\in\Omega$ and the time $t$, the quantity $\rho$ varies under the law
\begin{equation}\label{trasportodiquestocap} \begin{cases}\rho_t(x,t)-u(x,t)\cdot \nabla \rho(x,t)=0\\
\rho(x,0)=\rho_0(x)
\end{cases}\end{equation}
that describes the evolution of the initial density $\rho_0$ via transport equation.
\paragraph{Classical theory} The transport equation in a bounded domain has a well-known regular theory, that assures the existence and uniqueness of the solution in the case of a regular velocity field $u=u(x,t)$ and a regular initial density $\rho_0$. In particular, solving the equation via characteristics method, it is straighforward that the initial values provide naturally upper and lower bound to the solution at every time. This facts are summarized by the following theorem, see \cite{Figueredo:2009dg2}.
\begin{theorem}\label{maintransportequation} Let $\Omega$ be a bounded domain. Let $u(x,t)\in C([0,T]; C^1(\overline\Omega)^N)$ with $\nabla \cdot u=0$ and $u(x,t)=0$ for every $(x,t)\in\partial\Omega\times[0,T]$. Let $\rho_0\in C^1(\overline\Omega;\R)$. Then the problem \eqref{trasportodiquestocap} has a unique solution $\rho\in C^1([0,T]\times\overline\Omega;\R)$. Moreover: \begin{enumerate}[label=(\roman*)]
\item if $\alpha,\beta\in \R$ are such that $\alpha\leq \rho_0(x)\leq\beta$ for every $x\in\overline\Omega$, then
$$\rho(x,t)\in [\alpha,\beta] \qquad \forall (x,t)\in [0,T]\times\overline\Omega$$
\item The density solution $\rho$ satisfies a mass incompressibility property, that is $\|\rho(t)\|_q=\|\rho_0\|_q$, for every $q\in[1,\infty]$ and $t\in [0,T]$.
\end{enumerate}
\end{theorem}
The necessity to consider a weaker formulation of the problem is due, among the other things, to the presence of this equation in the Navier-Stokes (density dependent) system, where the weak interpretation of the equations plays an important role, in order to study the well-posedness problem. 

In this article we follow the seminal work \cite{Figueredo:2009dg} by DiPerna-Lions, to prove existence and uniqueness of weak solution to the transport equation, together with a fundamental stability theorem. DiPerna-Lions' work studies the transport equation in the whole space $\R^N$. However, in their work there is no mention of the bounded domain case. In the following, we will uniquely inspect this "new" case. In order to trace the hypothesis of the classical theory in a bounded domain, we will suppose that the velocity field is free-divergence and with zero boundary conditions, in weak (trace) sense. These hypothesis will be helpful in a moment, since we hope to apply a limit argument to the classical solutions. We will avoid the formalism of weak theory, always writing the distributional interpretation of the weak derivatives in the (rigorous) integral form.

The paper is concretely inspired by few lines in \cite{kimchoe1}, where the results of the work by DiPerna-Lions are applied to deduce a weak-* convergence. Through this paper, it will seem that the hypothesis on the initial data are a bit relaxed: in the stability theorem we will ask the velocity field to be continuous on $\overline\Omega$. In the context of the weak formulation theory of the $3$-dimensional Navier-Stokes system, this request is ensured by the typical hypothesis of the weak velocity field in the weak-strong theory, for example $u\in H^2(\Omega)$, that in $N=3$ assures the continuity until the boundary of $\Omega$ (see classical results, e.g. \cite{evans10}).

\subsection{Notations and preliminaries}

Given a domain $\Omega\subseteq\R^N$, we indicate with $W^{k,q}(\Omega)$ the standard Sobolev space over $\Omega$. The closure of the test functions in the topology of this space is $W_0^{k,q}(\Omega)$. For a function $v\in W^{k,q}(\Omega)$, we say that $\nabla \cdot v=0$, i.e. $v$ is divergence-free in the weak sense, if 
\begin{equation}
\int_\Omega v\cdot \nabla \varphi \ dx=0
\end{equation}
for every $\varphi\in C_c^\infty(\Omega)$. We set $W_{0,\sigma}^{k,q}(\Omega):=\{v\in W_0^{k,q}(\Omega): \ \nabla\cdot v=0\}$. This space, equipped with the norm $\|\cdot\|_{W^{k,q}(\Omega)}$, is a Banach space. It will be useful the following theorem that collects various standard results about the convergence in measure.
\begin{theorem}\label{propositionconvmeas31} Let $(\Omega,\mathcal{M},\mu)$ be a measure space, with $\mu(\Omega)<\infty$. Let $f_n,f$ be measurable functions over $\Omega$. Then the following properties hold.
\begin{enumerate}[label=(\roman*)]
\item If $f_n\to f$ in measure and exists $g\in L^p(\Omega)$ such that $|f_n|\leq g$, then $f_n\to f$ in $L^p(\Omega)$.
\item If $f_n\to f$ in measure and $\beta$ is a continuous function over $\R$, then $\beta(f_n)\to \beta(f)$ in measure.
\item Let $f_n$ a sequence of measurable functions, such that for every $\beta_k$  piecewise differentiable such that 
$$\beta_k(t):=\begin{cases}
\beta_k(t)=0 & |t|\leq \frac1k\\
\beta_k'(t)>0 & |t|>\frac1k\\
\beta_k, \ \beta_k' \ \text{are bounded}
\end{cases}$$
it exists $v_k$ measurable function such that 
$$\beta_k(f_n)\to v_k \qquad \text{in measure as $n\to\infty$}$$
If moreover $f_n\in L^p(\Omega)$, with $\displaystyle \sup_{n\in\N}\|f_n\|_{L^p(\Omega)}<\infty$, it follows that exists $f$ measurable function such that 
$$f_n\to f \qquad \text{in measure as $n\to\infty$}$$
\end{enumerate}
\end{theorem}
We also have to remember the following theorem about uniform convergence in time-dipendent Banach spaces (a Banach spaces version of Ascoli-Arzelà's theorem).
\begin{theorem}\label{lemmaan194} Let $X$ be a Banach space, and let $-\infty<a<b<\infty$. Let $f_n\in C([a,b];X)$ be a sequence such that, for every $t_0\in [a,b]$ and for every $[a,b]\ni t_n\to t_0$ 
\begin{equation}\label{limitconditionfntn94}\lim_{n\to\infty}\|f_n(t_n)-f(t_0)\|_X=0\end{equation}
with $f\in C([a,b];X)$. Then $f_n\to f$ in $C([a,b];X)$.
\end{theorem}
This easy integral estimate will be useful in the last part of the paper.
\begin{lemma}\label{teoremaunifintgfn1a} Let $p\in (1,\infty)$. Let $f_n\in L^p(\Omega)$ a sequence of function in $L^p(\Omega)$ such that $\displaystyle \sup_{n\in\N}\|f_n\|_p<\infty$. Then, for every $\varepsilon>0$ exists $M_\varepsilon>0$ such that 
\begin{equation}\sup_{n\in\N}\left\{\int_{\{x\in\Omega: \ |f_n(x)|>M_\varepsilon\}}|f_n(x)| \ dx\right\}<\varepsilon\end{equation}
\end{lemma}
\section{Weak theory}
\begin{definition}[Linear transport equation] Let $\Omega$ be a bounded domain in $\R^N$ and $T>0$. Consider $p\in [1,\infty]$ and let $q$ be its conjugate, such that $\frac1p+\frac1q=1$. Let $u\in L^1(0,T;W_0^{1,q}(\Omega))$ be a velocity field over $(0,T)\times\Omega$, with $\nabla \cdot u=0$, i.e. $u$ satisfies the divergence-free property, in the weak sense. Let $\rho_0\in L^p(\Omega)$ be the initial density. We say that \textit{the density $\rho\in L^\infty(0,T;L^p(\Omega))$ satisfies the (weak) transport equation}  
\begin{equation}\label{trasportodiqua}\begin{cases}\rho_t-u\cdot \nabla \rho=0 \quad & \text{in $(0,T)\times\Omega$}\\
\rho(0)=\rho_0
\end{cases}\end{equation}
if it is a solution of \eqref{trasportodiqua} in distributional sense, that is 
\begin{equation}\label{trasportodistirbuzionale}-\int_0^T\bigg(\int_\Omega \rho\  \phi_t \ dx\bigg) \ dt -\int_\Omega \rho_0(x)\phi(0,x)\ dx+\int_0^T\bigg(\int_\Omega \rho\  (u\cdot \nabla\phi) \ dx\bigg) \ dt=0\end{equation}
for every test function $\phi\in C^\infty([0,T]\times \Omega;\R)$ with compact support in $[0,T)\times\Omega$. This space can also be denoted by $\mathcal{D}([0,T)\times\Omega)$ or $C_c^\infty([0,T)\times\Omega;\R)$.
\end{definition}
We first prove an existence theorem.
\begin{theorem}[Existence of weak solutions] \label{exsolweak5} Let $p\in(1,\infty]$, $\rho_0\in L^p(\Omega)$. Let $q$ be the conjugate exponent of $p$. Suppose that $u\in L^1(0,T;W_0^{1,q}(\Omega))$, with $\nabla \cdot  u=0$. Then, there exists a solution of \eqref{trasportodiqua} in $L^\infty(0,T;L^p(\Omega))$ with initial density $\rho_0$.
\end{theorem}
\begin{figure}[tb]
\centering 
\includegraphics[width=0.25\columnwidth]{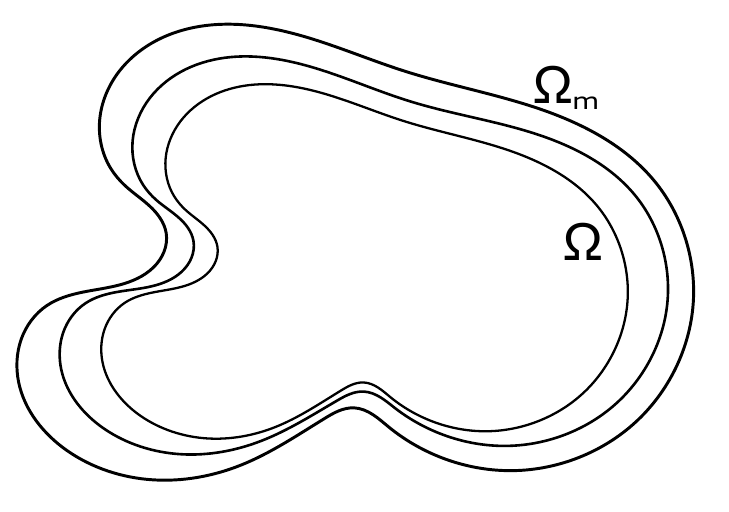} 
\caption[An example of a floating figure]{The domains $\Omega_m$ approach $\Omega$ from the outside.} 
\label{fig:gallery} 
\end{figure}
\begin{proof}
The proof that we present here is based over a classical regularization argument. Using the density of test function in Lebesgue spaces we can find a sequence $u^n\in C_c^\infty(0,T;W_{0,\sigma}^{1,q}(\Omega))$ such that 
\begin{equation}\label{convfortspbc1q0dv}\lim_{n\to\infty}\|u-u^n\|_{L^1(0,T;W_{0,\sigma}^{1,q}(\Omega))}=0\end{equation}
Since $u^n(t)\in W_0^{1,q}(\Omega)$, each element of the sequence can be extended to be zero outside $\Omega$. Moreover the initial density $\rho_0$ can be approached in $L^p(\Omega)$ with a sequence $\rho_n^0\in C_c^\infty(\Omega)$. We now set
\begin{equation} A_m:=\{x\in \Omega^c: \ \text{dist}(x,\partial\Omega)>\frac1m\}, \qquad \Omega_m:=A_m^c\end{equation}
See figure~\vref{fig:gallery}. 
We define 
\begin{equation}u^{m,n}(x,t):=\int_{\Omega_m} \eta_m(x-y) u^n(y,t) \ dy\end{equation}
For every $t\in [0,T]$ fixed, this convolution is smooth in $x\in\Omega_m$. Moreover, it is continuous as a function of two variables. 
Thanks to the convolution properties, the $x$-derivative is continuous over $\overline\Omega_m$: in particular, the gradient $\nabla u^{m,n}$ has the same integral form of $u^{m,n}$. So $u^{m,n}\in C([0,T];C^1(\overline\Omega_m))$. We have, furthermore, a gradient estimate:
\begin{equation}\label{stimacontrolloconvoluzione}|\nabla u^{m,n}(x,t)|=\bigg|\int_{\Omega_m}\eta_m(x-y) \nabla u^n(y,t)\ dy\bigg|\leq \bigg(\int_{\Omega_m}|\eta_m(x-y)|^p\ dy\bigg)^\frac1p\|\nabla u^n(\cdot, t)\|_q\leq\end{equation}
$$\leq\bigg(\int_{\R^N}|\eta_m(x-y)|^p\ dy\bigg)^\frac1p\max_{t\in[0,T]}\|\nabla u^n(\cdot, t)\|_q\equiv \bigg(\int_{\R^N}|\eta_m(z)|^p\ dz\bigg)^\frac1p\max_{t\in[0,T]}\|\nabla u^n(\cdot, t)\|_q$$
so that 
$$\sup_{t\in [0,T]}\|\nabla u^{m,n}(t)\|_\infty\leq \bigg(\int_{\R^N}|\eta_m(z)|^p\ dz\bigg)^\frac1p\max_{t\in[0,T]}\|\nabla u^n(\cdot, t)\|_q$$
We observe two properties of this approximation: 
\begin{enumerate}[label=(\roman*)]
\item If $x\in\partial \Omega_m$, we have $\displaystyle u^{m,n}(x,t)=\int_{\Omega_m}\eta_m(x-y)u^n(y,t)\ dy=0$
since $u^n(y,t)=0$ if $y\in B(x,\frac1m)$;
\item Moreover, 
\begin{equation}\nabla \cdot u^{m,n}(x,t)=\int_{\Omega_m}\eta_m(x-y)\nabla\cdot u^n(y,t)\ dy=0\end{equation}
since $\nabla \cdot u^n(y,t)=0$ by the definition of $u^n$. 
\end{enumerate}
\begin{figure}[tb]
\centering 
\includegraphics[width=0.4\columnwidth]{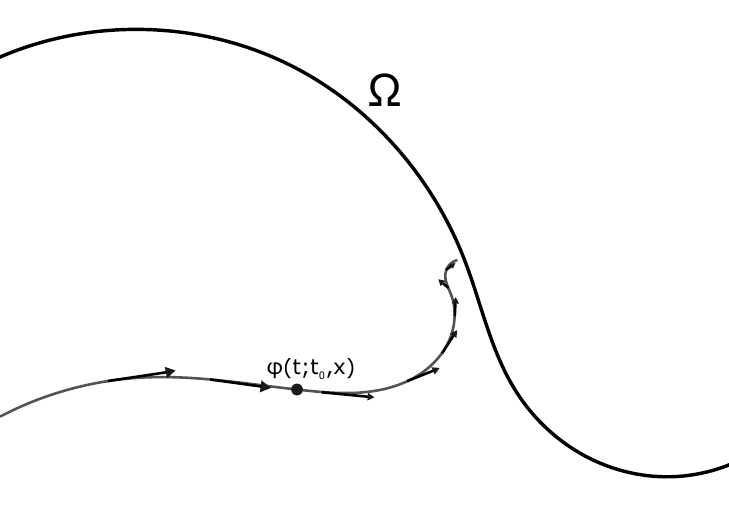} 
\caption[An example of a floating figure]{The regularized velocity fields have equilibrium points at the boundary, and thus the trajectories with starting point at the boundary remain constantly on this point for every time, while the interior trajectories, in example the flow $\varphi(t;t_0,x)$, never reach the boundary.} 
\label{fig:boundary} 
\end{figure}
So, we can use this velocity field to solve the transport problem 
$$\begin{cases}\rho_t-u^{m,n}\cdot \nabla \rho=0 & \text{in $[0,T]\times\overline\Omega_m$}\\
\rho(0,x)=\rho_0^n 
\end{cases}$$
\begin{remark} The zero boundary condition and the divergence-free condition are fundamental to apply Theorem \ref{maintransportequation}. As we will see in a moment, we need these conditions to avoid that the (regular) characteristic curves escape the domain $\Omega$, as we send them to limit in order to obtain weak solutions. See figure ~\vref{fig:boundary}.
\end{remark}
We name $\rho^{m,n}$ the solution of this classical transport equation. We know, according to the classical theory presented above, that, eventually renaming the sequence, 
$$\|\rho^{m,n}(t)\|_p=\|\rho_0^n\|_p\leq \|\rho_0\|_p+1=: C_0$$
It follows that $\|\rho^{m,n}\|_{L^\infty(0,T;L^p(\Omega))}\leq C_0$. Observe that, since $p\in (1,\infty]$, $L^\infty(0,T;L^p(\Omega))\simeq(L^1(0,T;L^q(\Omega))^*$, where $q$ is such that $\frac1p+\frac1q=1$. Moreover $L^q(\Omega)$ is separable, since $q\in [1,\infty)$. So, $L^1(0,T;L^q(\Omega))$ is separable. Then, thanks to the sequential version of Hanh-Banach theorem, we have that exists a weak-star converging subsequence, that approaches to some $\rho^n\in L^\infty(0,T;L^p(\Omega))$, that is 
\begin{equation}\label{weakstarconvergencerho37}\rho^{m_k,n}\stackrel{*}{\rightharpoonup} \rho^n \end{equation}
in $L^\infty(0,T;L^p(\Omega))\simeq(L^1(0,T;L^q(\Omega))^*$. In particular, the sequence satisfies
\begin{equation}\label{weaktranspformapprox1}-\int_\Omega(\rho^{m_k,n}\varphi)(0)\ dx-\int_0^{T}\int_\Omega\rho^{m_k,n} \varphi_t\ dx \ dt=\int_0^{T}\int_\Omega\rho^{m_k,n} u^{m_k,n}\cdot \nabla \varphi\ dx\ dt\end{equation}
for every $\varphi\in C_c^\infty(\Omega\times [0,T);\R)$, since $\rho^{m,n}$ is a classical solution and $u^{m,n}$ has free-boundary condition. Our aim is to pass to the limit the equation \eqref{weaktranspformapprox1}.  Observe, first of all, that 
\begin{equation}\int_\Omega \big(\rho^{m_k,n}\varphi\big)(0) \ dx\equiv \int_\Omega \rho_0^n(x)\varphi(x,0) \ dx,\qquad \int_0^T\int_\Omega \rho^{m_k,n}\varphi_t \ dx \ dt\stackrel{k\to\infty}{\longrightarrow} \int_0^T \int_\Omega \rho^n \varphi_t \ dx\ dt\end{equation}
thanks to the weak-$*$ convergence \eqref{weakstarconvergencerho37}. Furthermore 
$$\bigg|\int_0^{T}\int_\Omega\rho^{m_k,n} u^{m_k,n}\cdot \nabla \varphi\ dx\ dt-\int_0^{T}\int_\Omega\rho^{n} u^{n}\cdot \nabla \varphi\ dx\ dt\bigg|\leq$$
$$\leq \bigg|\int_0^T\int_\Omega (\rho^{m_k,n}-\rho^{n})u^{n}\cdot \nabla \varphi \ dx \ dt\bigg| + C\bigg(\int_0^T\|\rho^{m_k,n}\|_p\|u^n-u^{m_k,n}\|_q \ dt\bigg)\leq $$
$$\leq \bigg|\int_0^T\int_\Omega (\rho^{m_k,n}-\rho^{n})u^{n}\cdot \nabla \varphi \ dx \ dt\bigg| + C\bigg(\sup_{(0,T)}\|\rho^{m_k,n}\|_p\bigg)\bigg(\int_0^T\|u^n-u^{m_k,n}\|_q \ dt\bigg)\leq$$
\begin{equation}\leq \bigg|\int_0^T\int_\Omega (\rho^{m_k,n}-\rho^{n})u^{n}\cdot \nabla \varphi \ dx \ dt\bigg| + CC_0\|u^n-u^{m_k,n}\|_{L^1(0,T;L^q(\Omega)}\end{equation}
where $C$ is an upper bound for the derivative of the test function $\nabla\varphi$. Observe now that $\|u^n-u^{m_k,n}\|_{L^1(0,T;L^q(\Omega))}\to0$ as $k\to\infty$, thanks to \eqref{convfortspbc1q0dv}, and, moreover $\displaystyle \int_0^T\|u^n\cdot \nabla \varphi\|_q \ dt\leq C\int_0^T\|u^n\|_q\ dt<\infty$
that is, $u^n\cdot \nabla \varphi\in L^1(0,T;L^q(\Omega))$ and so the weak star convergence of $\rho^{m_k,n}$ implies that 
\begin{equation}\int_0^T\int_\Omega (\rho^{m_k,n}-\rho^{n})u^{n}\cdot \nabla \varphi \ dx \ dt\stackrel{k\to\infty}{\longrightarrow} 0\end{equation}
It follows that equation \eqref{weaktranspformapprox1}, sent to the limit, becomes
\begin{equation}\label{limitsendrhon1a}-\int_\Omega\rho_0^{n}(x)\varphi(x,0)\ dx-\int_0^{T}\int_\Omega\rho^{n} \varphi_t\ dx \ dt=\int_0^{T}\int_\Omega\rho^{n} u^{n}\cdot \nabla \varphi\ dx\ dt\end{equation}
Moreover, by the weak-$*$ convergence property, we have 
\begin{equation}\label{boundrhonc0}\|\rho^n\|_{L^\infty(0,T;L^p(\Omega))}\leq \liminf_{k\to\infty}\|\rho^{m_k,n}\|_{L^\infty(0,T;L^p(\Omega))}\leq C_0\end{equation}
We now want to send $n\to\infty$ in \eqref{limitsendrhon1a}. Clearly, if $C$ is an upper bound of $\varphi$,
\begin{equation}\bigg|\int_\Omega (\rho_0^n(x)-\rho^0(x))\varphi(x,0) \ dx\bigg|\leq C\|\rho_0^n-\rho^0\|_p\to0\end{equation}
By the bound \eqref{boundrhonc0}, we have that there exists a subsequence $n_h$ and $\rho \in L^\infty(0,T;L^p(\Omega))$ such that, as $h\to\infty$,
\begin{equation}\rho^{n_h}\stackrel{*}{\rightharpoonup} \rho \end{equation}
It follows that 
$$\bigg|\int_0^{T}\int_\Omega\rho^{n_h} u^{n_h}\cdot \nabla \varphi\ dx\ dt-\int_0^{T}\int_\Omega\rho u\cdot \nabla \varphi\ dx\ dt\bigg|=\bigg|\int_0^T\int_\Omega (\rho -\rho^{n_h})u\cdot \nabla \varphi \ dx \ dt -\int_0^T\int_\Omega \rho^{n_h} (u^{n_h}-u)\cdot \nabla \varphi \ dx \ dt\bigg|\leq$$
$$\leq \bigg|\int_0^T\int_\Omega (\rho -\rho^{n_h})u\cdot \nabla \varphi \ dx \ dt\bigg|+C\int_0^T\|\rho^{n_h}\|_p\|u^{n_h}-u\|_q \ dt\leq \bigg|\int_0^T\int_\Omega (\rho -\rho^{n_h})u\cdot \nabla \varphi \ dx \ dt\bigg|+$$
\begin{equation}+C\bigg(\sup_{(0,T)}\|\rho^{n_h}\|_p\bigg)\int_0^T\|u^{n_h}-u\|_q \ dt\leq \bigg|\int_0^T\int_\Omega (\rho -\rho^{n_h})u\cdot \nabla \varphi \ dx \ dt\bigg|+CC_0\|u^{n_h}-u\|_{L^1(0,T;L^q(\Omega))}\end{equation}
Since $u\cdot \nabla \varphi \in L^1(0,T;L^q(\Omega))$ and since $\rho^{n_h}$ converges weakly star to $\rho$, we have
\begin{equation}\int_0^T\int_\Omega (\rho -\rho^{n_h})u\cdot \nabla \varphi \ dx \ dt\stackrel{k\to\infty}{\longrightarrow} 0\end{equation}
It follows that 
\begin{equation}-\int_\Omega\rho^0(x)\varphi(x,0)\ dx-\int_0^{T}\int_\Omega\rho \varphi_t\ dx \ dt=\int_0^{T}\int_\Omega\rho u\cdot \nabla \varphi\ dx\ dt\end{equation}
So we have found $\rho\in L^\infty(0,T;L^p(\Omega))$ such that it is a weak solution to the trasport equation with velocity $u$ and initial density $\rho^0$. This proves the theorem.
\end{proof}
\subsection{Weak solutions and convolutions}
We now observe that, under suitable hypothesis on $u$, weak solution to equation \eqref{trasportodiqua} can be approached by smooth (in space) solution of \eqref{trasportodiqua}, plus an error term. In particular, we have the following theorem. 
\begin{theorem}[Convolution in $x$ of weak solutions]\label{approxregularixazione} Let $\Omega$ be a bounded domain. Consider $p\in(1,\infty]$, and $\rho\in L^\infty(0,T;L^p(\Omega))$ a solution of \eqref{trasportodiqua} with initial density $\rho_0\in L^p(\Omega)$ and assume that $u\in L^1(0,T;W^{1,\alpha}(\Omega))$ for some $\alpha\geq q$, $\nabla \cdot u=0$, where $q$ is the conjugate exponent of $p$. Let $\eta_\varepsilon=\eta_\varepsilon(x)$ be a regularizing kernel over $\Omega$. In particular, for every $\varepsilon>0$, if $\Omega_\varepsilon:=\{x\in\Omega: \ \text{dist}(x,\partial\Omega)>\varepsilon\}$, we define 
\begin{equation}\eta_\varepsilon(x):=\frac{1}{\varepsilon^{n}}\eta\bigg(\frac{x}{\varepsilon}\bigg)\end{equation}
with  $C_c^\infty(\R^{n};\R)\ni\eta\geq0$, $\text{supp}(\eta)\subset B(0,1)$. Let $\rho_\varepsilon(x,t):=\big(\rho(\cdot,t)*\eta_\varepsilon\big)(x,t)$. Let $\phi\in C_c^\infty([0,T)\times\Omega;\R)$ and suppose $\phi(x,\cdot)=0$ for every $x\in \Omega_0^c$, with $\Omega_0$ a compact subset of $\Omega$. Then, if $\varepsilon<\text{dist}(\Omega_0,\partial\Omega)$,
\begin{equation}\label{weakformaaproxeps96}-\int_0^T\bigg(\int_{\Omega_\varepsilon} \rho_\varepsilon \frac{\partial \phi}{\partial t} \ dx\bigg) \ dt-\int_{\Omega_\varepsilon} \rho_\varepsilon^0 \ \phi(0,x) \ dx +\int_0^T\bigg(\int_{\Omega_\varepsilon} \rho_\varepsilon u\cdot \nabla \phi \ dx\bigg) \ dt=\int_0^T\bigg(\int_{\Omega} r_\varepsilon\phi \ dx\bigg)\ dt
\end{equation}
where 
\begin{equation}\label{restoapprossimazione} r_\varepsilon(x,t):=\int_\Omega \rho(y,t)(u(y,t)-u(x,t))\cdot \nabla \eta_\varepsilon(y-x) \ dy, \qquad \rho_\varepsilon^0(x):=(\rho_0*\eta_\varepsilon)(x)\end{equation}
Moreover, for every $\Omega_0\subseteq\Omega$, $r_\varepsilon$ goes to zero in $L^1(0,T;L^\gamma(\Omega_0))$ when $\varepsilon\to0$,  
where $\gamma$ is such that 
\begin{equation}\frac{1}{\gamma}=\frac{1}{\alpha}+\frac1p\end{equation}
\end{theorem}
\begin{remark}
The convergence to zero of $r_\varepsilon$ in $L^1(0,T;L^\gamma(\Omega_0))$ assures that
$$\bigg|\int_0^T \bigg(\int_\Omega r_\varepsilon \phi \ dx\bigg) \ dt\bigg|=\bigg|\int_0^T \bigg(\int_{\Omega_0} r_\varepsilon \phi \ dx\bigg) \ dt\bigg|\leq |\Omega|^{\frac{\gamma-1}{\gamma}}\bigg(\sup_{[0,T]\times\Omega}|\phi|\bigg)\int_0^T\|r_\varepsilon\|_{L^\gamma(\Omega_0)} \ dt\to0$$
for $\varepsilon\to 0$. 
\end{remark}
\begin{proof} 
The proof that the term \eqref{restoapprossimazione} goes to zero in the suitable norm as $\varepsilon\to0$ is based on the same arguments in \cite{Figueredo:2009dg}, with analogous calculations. We only remark that 
$$\int_0^T\bigg(\int_{\Omega_\varepsilon} \rho_\varepsilon(x,t)\frac{\partial\phi}{\partial t}(x,t)\ dx\bigg)\ dt=\int_0^T\bigg\{\int_{\Omega_\varepsilon} \bigg(\int_\Omega \rho(y,t) \eta_\varepsilon(x-y) \ dy\bigg)\frac{\partial\phi}{\partial t}(x,t)\ dx\bigg\}\ dt=$$
$$=\int_0^T\bigg\{\int_\Omega \bigg(\int_{\Omega_\varepsilon} \eta_\varepsilon(x-y) \frac{\partial \phi}{\partial t}(x,t) \ dx\bigg)\ \rho(y,t) \ dy\bigg\}\ dt=\int_0^T\bigg\{\int_\Omega \frac{\partial}{\partial t}\phi_\varepsilon(y,t) \ \rho(y,t) \ dy\bigg\}\ dt$$
since $\eta_\varepsilon(x-y)=\eta_\varepsilon(y-x)$ by definition, and, being $\varepsilon<\text{dist}(\Omega_0,\partial\Omega)$, we have $\phi(x,t)\equiv0$ in $\Omega/\Omega_\varepsilon$, so that 
$$\int_{\Omega_\varepsilon} \eta_\varepsilon(x-y) \frac{\partial \phi}{\partial t}(x,t) \ dx=\int_{\Omega}\eta_\varepsilon(x-y) \frac{\partial \phi}{\partial t}(x,t) \ dx=:\frac{\partial}{\partial t}\phi_\varepsilon(y,t)$$
In the same way, we have 
$$\int_{\Omega_\varepsilon} \rho_\varepsilon^0(x)\phi(0,x) \ dx=\int_\Omega \phi_\varepsilon(0,y) \rho^0(y) \ dy$$
Analogously  
$$\int_0^T\bigg(\int_{\Omega_\varepsilon} \rho_\varepsilon(x,t) u(x,t)\cdot \nabla \phi(x,t) \ dx\bigg) \ dt=\int_0^T\bigg\{\int_\Omega \rho(y,t) \ \bigg(\int_{\Omega_\varepsilon}\eta_\varepsilon(x-y) u(x,t)\cdot \nabla \phi(x,t) \ dx\bigg) \ dy\bigg\} \ dt=$$
and being $\phi(x,t)\equiv0$ on $\Omega_\varepsilon^c$, and since $\nabla \eta_\varepsilon(x-y)=-\nabla \eta_\varepsilon(y-x)$, and $\nabla\cdot u=0$,
$$=\int_0^T\bigg\{\int_\Omega \rho(y,t)\bigg(\int_\Omega \eta_\varepsilon(x-y) u(x,t)\cdot \nabla \phi(x,t) \ dx \bigg) \ dy\bigg\}\ dt=\int_0^T\bigg\{\rho(y,t)\bigg(\int_\Omega \phi(x,t) u(x,t)\cdot\nabla \eta_\varepsilon(y-x) \ dx\bigg)\ dy\bigg\} \ dt$$
\end{proof}
The riformulation obtained in theorem \ref{approxregularixazione} allows to prove important results concerning the weak transport equation. In particular, we now prove a theorem that paves the way to the proof of the uniqueness of the (weak) solution. It also introduces the concept of renormalized solution. 
\begin{theorem}\label{lemmaapproxbeta} Let $\Omega$ be a bounded domain. Fix $p\in(1,\infty]$, and consider $\rho\in L^\infty(0,T;L^p(\Omega))$, a solution of  \eqref{trasportodiqua} with initial density $\rho_0\in L^p(\Omega)$ and assume that $u\in L^1(0,T;W^{1,\alpha}(\Omega))$ for some $\alpha\geq q$, $\nabla \cdot u=0$. Let $\eta_\varepsilon=\eta_\varepsilon(x)$ be a regularizing kernel over $\Omega$. In particular, if $\Omega_\varepsilon:=\{x\in\Omega: \ \text{dist}(x,\partial\Omega)>\varepsilon\}$, define 
$$\eta_\varepsilon(x):=\frac{1}{\varepsilon^{n}}\eta\bigg(\frac{x}{\varepsilon}\bigg)$$
with $C_c^\infty(\R^{n})\ni\eta\geq0$, $\text{supp}(\eta)\subset B(0,1)$. Let  $\rho_\varepsilon(x,t):=\big(\rho(\cdot,t)*\eta_\varepsilon\big)(x,t)$. Let $\phi\in C_c^\infty([0,T)\times\Omega)$ and suppose that $\phi(x,\cdot)=0$ for every $x\in \Omega_0^c$, with $\Omega_0$ compact set. Let $\beta\in C^1(\R)$ a function, with $\beta,\beta'$ bounded.
Then, if $\varepsilon<\text{dist}(\Omega_0,\partial\Omega)$, equation \eqref{weakformaaproxeps96} holds, and
$$-\int_0^T\bigg(\int_{\Omega_\varepsilon} \beta(\rho_\varepsilon) \frac{\partial \phi}{\partial t} \ dx\bigg) \ dt-\int_{\Omega_\varepsilon} \beta(\rho_\varepsilon^0) \ \phi(0,x) \ dx +\int_0^T\bigg(\int_{\Omega_\varepsilon} \beta(\rho_\varepsilon) u\cdot \nabla \phi \ dx\bigg) \ dt=$$
$$=\int_0^T\bigg(\int_{\Omega_\varepsilon} r_\varepsilon\beta'(\rho_\varepsilon)\phi \ dx\bigg)\ dt$$
where, as above,
$$r_\varepsilon(x,t)=\int_\Omega \rho(y,t)(u(y,t)-u(x,t))\cdot \nabla \eta_\varepsilon(y-x) \ dy$$
\end{theorem}
\begin{remark} In this theorem, as in the previous one, we prove \textit{a posteriori} results about a solution that we already know that exists. So, we only require that $u\in L^1(0,T;W^{1,\alpha}(\Omega))$, without any request at the boundary. We know, however, that in order to assure the existence of a solution, it is required a zero boundary condition, by the first theorem proved. \end{remark}
\begin{proof} Consider \eqref{weakformaaproxeps96},
and choose now $\phi(x,t):=\varphi(x)\psi(t)$, with $\psi\in C_c^\infty(0,T)$ and $\varphi$ to be fixed. Then we have, using that $\rho_\varepsilon u\phi\equiv0$ on $\partial\Omega_\varepsilon$ and the divergence theorem,
$$-\int_0^T\bigg(\int_{\Omega_\varepsilon}\rho_\varepsilon(x) \psi'(t) \varphi(x) \ dx\bigg) \ dt=\int_0^T\psi(t)\bigg(\int_{\Omega_\varepsilon}(u\cdot \nabla \rho_\varepsilon+r_\varepsilon)(x) \ \varphi(x)\bigg) \ dt$$
If we choose  $\varphi$ as the unitary mass sequence $\varphi_y^m(x):=\eta_{\frac1m}(y-x)$, concentrated in $y$, it follows 
in the sense of weak derivatives that
\begin{equation}\label{formulazionedobole28}(\rho_\varepsilon)_t(y,t)=u(y,t)\cdot \nabla \rho_\varepsilon(y,t)+r_\varepsilon(y,t)\end{equation}
since $\displaystyle \sup_{(0,T)}\|\rho_\varepsilon\|_\infty\leq \left(\int_{\R^N}|\eta_\varepsilon|^q\right)^\frac1q\sup_{(0,T)}\|\rho\|_p$, and $\displaystyle \sup_{(0,T)}\|\nabla \rho_\varepsilon\|_\infty\leq \varepsilon^{-1}C$, for some constant $C>0$ and $\varepsilon$ is fixed. The same bound holds for $r_\varepsilon$, while $\|u\|_q$ is integrabile in $t$. So the Lebesgue convergence theorem implies \eqref{formulazionedobole28}. 
\\
In particular, equation \eqref{formulazionedobole28} is true for every $y\in\Omega_\varepsilon$. Moreover, we have $\rho_\varepsilon(y,\cdot)\in W^{1,1}(0,T)$, using exactly the same bounds we have deduced in order to apply the Lebesgue dominate convergence (since in both cases we are estimating a time integral). On the other hand,
using the Lebesgue differentiation theorem, we have that for almost every $t_0\in(0,T)$, with $y\in \Omega_\varepsilon$ fixed,
\begin{equation}\label{abscont910a}\rho_\varepsilon(y,t_0)=\rho_\varepsilon^0(y)+\int_0^{t_0}\big(u\cdot \nabla \rho_\varepsilon+r_\varepsilon\big)(y,t) \ dt\end{equation}
In particular, the right-side is a continuous version of $\rho_\varepsilon(y,\cdot)$. This means that $\rho_\varepsilon(y,\cdot)$ is absolutely continuous. Consider now $\beta\in C^1(\R)$ with $\beta'$ bounded. The weak chain rule says that 
$$(\beta(\rho_\varepsilon))_t=\beta'(\rho_\varepsilon)(\partial \rho_{\varepsilon})_t=\beta'(\rho_\varepsilon)(u\cdot \nabla \rho_\varepsilon+r_\varepsilon)=u\cdot \nabla (\beta(\rho_\varepsilon))+\beta'(\rho_\varepsilon) r_\varepsilon$$
since $\rho_\varepsilon$ has classical regularity in space. So, in particular, being $\beta'$ bounded, $\beta(\rho_\varepsilon)\in W^{1,1}(0,T)$ and so, moreover,
\begin{equation}\beta(\rho_\varepsilon)(y,t)=\beta(\rho^0_\varepsilon)(y)+\int_0^t\big(u\cdot \nabla (\beta(\rho_\varepsilon))+\beta'(\rho_\varepsilon) r_\varepsilon\big)(y,\tau) \ d\tau\end{equation}
is its continuous version. Consider now $\phi\in C_c^\infty([0,T)\times\Omega)$, so that $\phi(T,x)=0$. We know that, by the product rule,
$\beta(\rho_\varepsilon)\phi\in W^{1,1}(0,T)$. Moreover  
$$0=\beta(\rho_\varepsilon(T))\phi(T)=\beta(\rho^0_\varepsilon)\phi(0)+\int_0^T(\beta(\rho_\varepsilon))_t \phi \ dt+\int_0^T\beta(\rho_\varepsilon) \phi_t \ dt$$
Then we have 
$$\int_0^T\bigg(\int_{\Omega_\varepsilon} \beta(\rho_\varepsilon) \phi_t \ dx\bigg) \ dt=\int_{\Omega_\varepsilon}\bigg(\int_0^T \beta(\rho_\varepsilon) \phi_t \ dt\bigg) \ dx=$$
$$=-\int_{\Omega_\varepsilon}\beta(\rho^0_\varepsilon)\phi(0) \ dx-\int_{\Omega_\varepsilon}\bigg(\int_0^T(u\cdot \nabla (\beta(\rho_\varepsilon))+\beta'(\rho_\varepsilon) r_\varepsilon) \phi \ dt \bigg) \ dx=$$
$$=-\int_{\Omega_\varepsilon}\beta(\rho^0_\varepsilon)\phi(0) \ dx+\int_0^T\bigg(\int_{\Omega_\varepsilon}\beta(\rho_\varepsilon) u\cdot \nabla \phi \ dx\bigg) \ dt-\int_0^T\bigg(\int_{\Omega_\varepsilon}\beta'(r_\varepsilon)r_\varepsilon \phi \ dx\bigg) \ dt$$
that is the thesis, using that $\nabla\cdot u=0$ and $\phi=0$ on the boundary of $\Omega_\varepsilon$, thanks to the choice of $\varepsilon$. 
\end{proof}
\subsection{Uniqueness}
We finally prove a uniqueness theorem.
\begin{theorem}[Uniqueness]\label{unicitalemmarho0} Let $\Omega$ be a bounded domain. Consider $p\in(1,\infty]$, and $\rho\in L^\infty(0,T;L^p(\Omega))$ a solution of \eqref{trasportodiqua} with initial condition $\rho^0\equiv0$, $u\in L^1(0,T;W_0^{1,q}(\Omega))\cap L^1(0,T;C(\overline\Omega))$ and $\nabla \cdot u=0$, where $q$ is the conjugate of $p$. Then, $\rho\equiv 0$.
\end{theorem}
\begin{proof} First of all suppose $p<\infty$. It is not restrictive, since then $\rho\in L^{\infty}(0,T;L^p(\Omega))$, and if $p'=\infty$ the conjugate is $q'=1$, and so $q>q'$, since $\frac1p+\frac1q=1$. So we can apply theorem \ref{lemmaapproxbeta}. Letting $\varepsilon\to0$ in the statement of theorem \ref{lemmaapproxbeta}, with $\beta \in C^1(\R)$ bounded, and with $\beta'$ bounded, we have that 
\begin{equation}\label{eqrifth95a}-\int_0^T\bigg(\int_{\Omega} \beta(\rho) \frac{\partial \phi}{\partial t} \ dx\bigg) \ dt-\int_\Omega \beta(\rho_0) \phi(x,0) \ dx +\int_0^T\bigg(\int_{\Omega} \beta(\rho) u\cdot \nabla \phi \ dx\bigg) \ dt=0\end{equation}
where we used that $r_\varepsilon\to 0$ in $L^1(0,T;L^1_{loc}(\Omega))$. Let now $M\in(0,\infty)$. We would choose $\beta(t):=(|t|^p\wedge M)$, where $a\wedge b:=\min\{a,b\}$. The function is clearly bounded, but it is not in $C^1(\R)$. However, it is possible to choose $\beta_k(t)$ a sequence such that $\beta_k\in C^1(\R)$ for every $k$, $\beta_k(t)\leq \beta(t)$ for every $k\in\mathbb{N}$ and $t\in \R$ and finally, for every $t\in \R$, $\beta_k(t)\leq \beta_{k+1}(t)$, with $\beta_k(t)\to \beta(t)$ as $k\to\infty$, for almost every $t\in \R$. So \eqref{eqrifth95a} implies that 
\begin{equation}\label{equationrefpaslim912a}-\int_0^T\bigg(\int_{\Omega} \beta_k(\rho) \frac{\partial \phi}{\partial t} \ dx\bigg) \ dt -\int_\Omega \beta_k(\rho_0) \phi(x,0) \ dx  +\int_0^T\bigg(\int_{\Omega} \beta_k(\rho) u\cdot \nabla \phi \ dx\bigg) \ dt=0\end{equation}
for every $k\in\mathbb{N}$. It is clear that $\beta_k(t)\leq \beta(t)\leq M$. We now focus the attention on the last term of \eqref{equationrefpaslim912a}. We now choose $\phi(x,t)=\psi(t)\varphi(x)$ in a precise way. In particular, we choose a sequence $\varphi=\varphi_h\in C_c^\infty(\Omega)$ such that $\varphi_h\equiv 1$ over $\Omega_{\frac1h}$, and $|\nabla\varphi_h|\leq 2h$. See \cite{gilbarg2015elliptic}. Then, fixed $\psi$, and defined $\phi_h:=\psi \varphi_h$, 
\begin{equation}
\left|\int_0^T\bigg(\int_{\Omega} \beta_k(\rho) u\cdot \nabla \phi \ dx\bigg) \ dt\right|\leq M\left(\sup_{t\in [0,T]}|\psi(t)|\right)\int_0^T 2h\int_{\Omega\setminus \Omega_{\frac1h}} |u| \ dx \ dt
\end{equation}
We remark that $u\in L^1(0,T;C(\overline\Omega))$ and there exists a constant $C>0$ such that $\displaystyle 2h\int_{\Omega\setminus\Omega_{\frac1h}}|u| \ dx\leq C\sup_{\overline\Omega}|u|$, where $C$ is such that $2h|\Omega\setminus \Omega_\frac1h|\leq C$. See the note below.\footnote{In fact, if $f(\varepsilon):=|\tilde \Omega_\varepsilon|$, where here $|\cdot|$ is the measure of the set, and $\tilde \Omega_\varepsilon:=\Omega\setminus\Omega_\varepsilon$, so that $|\tilde \Omega_\varepsilon|=|\Omega|-|\Omega_\varepsilon|$, then by the coarea formula
\begin{equation}
-\frac{d}{d\varepsilon}|\Omega_\varepsilon|=\int_{\partial\Omega_\varepsilon}d\sigma \ \Longrightarrow \ \frac{d}{d\varepsilon}|\tilde \Omega_\varepsilon|=-\frac{d}{d\varepsilon}|\Omega_\varepsilon|=\int_{\partial\Omega_\varepsilon}d\sigma
\end{equation}
so that $\displaystyle |\tilde\Omega_{\overline\varepsilon}|=\int_0^{\overline\varepsilon}|\partial\Omega_\varepsilon|_\sigma \ d\varepsilon$, where here $|\cdot|_\sigma$ is the surface measure. Clearly, $\displaystyle \sup_{\varepsilon\in I}|\partial\Omega_\varepsilon|_\sigma<\infty$, where $I$ is a small neighborhood of the origin, since $f$ is $C^1$ near the origin, extending the function to the negative values approaching the domain from the exterior. So $\displaystyle|\tilde \Omega_\varepsilon|\leq \varepsilon\left(\sup_{\delta\in I(0)}|\partial\Omega_\delta|_\sigma\right)$, for $\varepsilon<<1$. So the claim follows choosing $\varepsilon=\varepsilon_h=\frac1h$.} Since $\|u\|_\infty\in L^1((0,T))$, and $|u|\in C(\overline\Omega)$ for almost every $t\in (0,T)$, then $\displaystyle \lim_{h\to\infty} 2h\int_{\Omega\setminus \Omega_\frac1h} |u| \ dx=0$ for almost every $t\in (0,T)$, using the coarea formula. Then, for every $k\in\N$, and $\psi$ smooth, choosing $\phi=\phi_h$ and sending $h\to\infty$ in \eqref{equationrefpaslim912a} we have 
\begin{equation}
-\int_0^T\bigg(\int_{\Omega} \beta_k(\rho) \psi'(t) \ dx\bigg) \ dt -\int_\Omega \beta_k(\rho_0) \psi(0) \ dx =0
\end{equation}
We suppose now $\psi(0)=1$. Using again the boundedness of $\psi'$ and the fact that $\beta_k$ has been taken increasing, letting $k\to\infty$ we have, choosing $M=n\in\N$ fixed
$$-\int_0^T\psi'(t)\bigg(\int_\Omega |\rho|^p\wedge n \ dx\bigg) \ dt-\int_\Omega |\rho_0|^p\wedge n \ dx=0$$
Choosing now $\psi$ as an approximation of a Dirac-delta with mass in $t=t_0$ and $\psi(0)=1$, we find that exists $E_n\subseteq (0,T)$, $|E_n|=0$, such that for every $t_0\in (0,T)\setminus E_n$ 
\begin{equation}\label{integrat0eqtot1a}\int_\Omega |\rho(t_0)|^p\wedge n \ dx\equiv\bigg(\int_\Omega |\rho|^p\wedge n \ dx\bigg)(t_0)=\int_\Omega |\rho_0|^p\wedge n \ dx\end{equation}
Since the sequence $|\rho|^p\wedge n$ is increasing in $n$, and $|\rho|^p\wedge n\to |\rho|^p$ when $n\to\infty$, and \eqref{integrat0eqtot1a} holds for every $n\in\N$ and $\displaystyle t_0\in (0,T)/\bigcup_n E_n$, we have that for almost every $t_0\in (0,T)$ 
\begin{equation}\label{finalineq919a}\|\rho(t_0)\|_p=\|\rho_0\|_p\end{equation}
Since, by the hypothesis $\rho_0\equiv0$, this means that for almost every $t_0\in (0,T)$, $\rho(t_0)=0$ almost every $x\in \Omega$. This means that $\rho$ is zero in $L^\infty(0,T;L^p(\Omega))$, that is the thesis. 

\end{proof}
The next corollary follows from the proof of theorem \ref{unicitalemmarho0}.
\begin{theorem}\label{corollariodatoinizialeconservato1a} Let $\Omega$ be a bounded domain. Consider $p\in(1,\infty]$, and $\rho_0\in L^p(\Omega)$, $u\in L^1(0,T;C(\overline\Omega))$, such that $u\big|_{\partial\Omega}\equiv0$. Let $\rho$ a measurable function on $\Omega\times(0,T)$ such that, for every $\beta\in C^1(\R)$, with $\beta,\beta'$ bounded,
\begin{equation}-\int_0^T\bigg(\int_{\Omega} \beta(\rho) \frac{\partial \phi}{\partial t} \ dx\bigg) \ dt-\int_\Omega \beta(\rho_0) \phi(x,0) \ dx +\int_0^T\bigg(\int_{\Omega} \beta(\rho) u\cdot \nabla \phi \ dx\bigg) \ dt=0\end{equation}
for every $\phi \in C_c^\infty(\Omega\times[0,T))$. Then, for almost every $t_0\in (0,T)$ we have 
\begin{equation}
\|\rho(t_0)\|_p=\|\rho_0\|_p
\end{equation}
\end{theorem}
\begin{remark} The assumption on $\rho$ are very "weak" (only measurability is required).
Observe that in the proof above it is not used that $q$ is the conjugate of $p$ (if not in proving that the approximate integral equation can be send to limit, but this step is skipped in the present statement). The minimality of these hypothesis will be useful in a moment. \end{remark}
\section{Renormalized solutions}
In the previous section we studied the properties of weak solutions to the transport equation. However, there is another way to interpret solutions (that is equivalent to the one already defined, as we will see in a moment).  
\begin{definition}[Renormalized solutions] Let $\Omega$ be a bounded domain in $\R^N$, and $T>0$. Let $p\in (1,\infty]$, $q$ its conjugate and $\rho_0\in L^p(\Omega)$ an initial density. Let $u\in L^1(0,T;W_0^{1,q}(\Omega))$, $\nabla \cdot u=0$ a velocity field. We will say that $\rho\in L^\infty(0,T;L^p(\Omega))$  is a \textit{renormalized solution} of 
\begin{equation}\begin{cases} \rho_t-u\cdot \nabla \rho=0 & \text{in $(0,T)\times \Omega$}\\
\rho(0)=\rho_0
\end{cases}\end{equation}
if, for every $\beta\in C^1(\R;\R)$, with $\beta$ and $\beta'$ bounded, it holds 
\begin{equation}\label{renormalizedweek96}-\int_0^T \bigg(\int_\Omega  \beta(\rho) \frac{\partial \phi}{\partial t}\ dx\bigg) \ dt-\int_\Omega \beta(\rho_0(x))\phi(x,0) \ dx +\int_0^T\bigg(\int_\Omega \beta(\rho) u\cdot \nabla \phi \ dx\bigg) \ dt=0\end{equation}
for every $\phi\in C_c^\infty([0,T)\times\Omega)$. Such a function $\beta$ is said \textit{admissible function}, and we will write $\beta\in\mathcal{A}$, with $\mathcal{A}$ the \textit{set of admissible functions}. For every $\beta\in\mathcal{A}$ we define 
\begin{equation}C_\beta:=\sup_\R|\beta|+\sup_\R|\beta'|<\infty\end{equation}
\end{definition}
\begin{theorem}\label{lemma82weakren} Let $\Omega$ be a bounded domain in $\R^N$ and $T>0$ a positive time. Let $p\in (1,\infty]$ and $\rho\in L^\infty(0,T;L^p(\Omega))$ solution of \eqref{trasportodiqua} with initial density $\rho_0\in L^p(\Omega)$ and assume that $u\in L^1(0,T;W^{1,q}(\Omega))$, $\nabla \cdot u=0$. Then $\rho \in L^\infty(0,T;L^p(\Omega))$ is a renormalized solution to the problem for every admissible function $\beta$. 
\end{theorem}
\begin{proof} By theorem \ref{lemmaapproxbeta} we know that 
\begin{equation}-\int_0^T\bigg(\int_{\Omega_\varepsilon} \beta(\rho_\varepsilon) \frac{\partial \phi}{\partial t} \ dx\bigg) \ dt-\int_{\Omega_\varepsilon} \beta(\rho_\varepsilon^0) \ \phi(0,x) \ dx +\int_0^T\bigg(\int_{\Omega_\varepsilon} \beta(\rho_\varepsilon) u\cdot \nabla \phi \ dx\bigg) \ dt=\int_0^T\bigg(\int_{\Omega_\varepsilon} r_\varepsilon\beta'(\rho_\varepsilon)\phi \ dx\bigg)\ dt\end{equation}
with $r_\varepsilon\to 0$ in $L^1(0,T;L^\gamma_{loc}(\Omega))$, with $\displaystyle \frac1\gamma=\frac1q+\frac1p=1 \Longrightarrow \gamma=1$. So, letting $\varepsilon\to0$, being $\beta$ bounded and $|\beta'(\rho_\varepsilon)|\leq C_\beta$, we have that the thesis follows. 
\end{proof}
\subsection{Classical regularity}
Before introducing the \textit{stability problem}, we focus our attention to the regularity of the weak solution to the transport equation. 
\begin{theorem}[Continuity (in time) of the solution]\label{lemmacorollario} Let $p\in(1,\infty)$ and $\rho_0\in L^p(\Omega)$. Assume that $u\in L^1(0,T;W^{1,q}(\Omega))$ with $\nabla \cdot u=0$. Then $\rho\in C([0,T];L^p(\Omega))$.
\end{theorem}
\begin{proof} By equation \eqref{finalineq919a} we have that $\|\rho(t)\|_p$ has a continuous version $\|\rho(t)\|_p=\|\rho_0\|_p\in C([0,T];\R)$. If we show that, moreover, for every $[0,T]\ni t_n\to t_0\in [0,T]$, it holds 
\begin{equation}\label{weakconvergenceLpLqtimetn1}\lim_{n\to\infty}\int_\Omega \big(\rho(x,t_n)-\rho(x,t_0)\big)\cdot \varphi(x)\ dx=0 \qquad \forall \ \varphi\in L^q(\Omega)\end{equation}
that is $\rho(t_n)\rightharpoonup \rho(t_0)$ in $L^p(\Omega)$, or in other words $\rho(t_n)$ converges weakly to $\rho(t_0)$ in $L^p(\Omega)$. Since moreover $\|\rho(t_n)\|_p\equiv \|\rho(t_0)\|_p$, we have that 
\begin{equation}\lim_{n\to\infty}\|\rho(t_n)-\rho(t_0)\|_p=0\end{equation}
that is the continuity in $C([0,T];L^p(\Omega))$. So, we only have to prove \eqref{weakconvergenceLpLqtimetn1}. It proceeds as it follows. If in equation \eqref{trasportodistirbuzionale} we choose $\phi(x,t)=\psi(t)\varphi(x)$, we have 
\begin{equation}-\int_0^T\psi(t)\bigg(\int_\Omega \rho(x,t)\varphi(x) \ dx\bigg) \ dt -\int_\Omega \rho_0(x)\psi(0)\varphi(x)\ dx +\int_0^T\psi(t)\bigg(\int_\Omega \rho(x,t)\  \big(u(x,t)\cdot \nabla\varphi(x)\big) \ dx\bigg) \ dt=0\end{equation}
Choosing $\psi(t)$ as a Dirac-delta as before, we have, for almost every $t_0\in [0,T]$, 
\begin{equation}\label{reformofeq812813}\int_\Omega \rho(x,t_0)\varphi(x) \ dx=\int_\Omega \rho_0(x) \varphi(x) \ dx-\int_0^{t_0}\bigg(\int_\Omega \rho(x,t)\big(u(x,t)\cdot \nabla\varphi(x)\big) \ dx\bigg) \ dt\end{equation}
The continuity of the right side implies that $\displaystyle \int_\Omega \rho(x,t_0)\varphi(x) \ dx$ can be defined in the whole $[0,T]$. Consider now $h>0$. Then, for every $\varphi\in C_c^\infty(\Omega)$ we have 
\begin{small}
$$\bigg|\int_\Omega \rho(x,t_0+h)\varphi(x) \ dx-\int_\Omega \rho(x,t_0)\varphi(x) \ dx\bigg|=\bigg|\int_{t_0}^{t_0+h}\bigg(\int_\Omega \rho(x,t)\big(u(x,t)\cdot \nabla \varphi(x)\big) \ dx\bigg) \ dt\bigg|\leq M_\varphi\int_{t_0}^{t_0+h}\|\rho(t)\|_p\|u(t)\|_q \ dt$$
\end{small}
where $\displaystyle M_\varphi:=\max_{\Omega}|\nabla\varphi|$. Since $\|\rho(t)\|_p\in L^\infty((0,T))$ and $\|u(t)\|_q\in L^1((0,T))$, by the Lebesgue dominated convergence it follows that, for every $\varphi\in C_c^\infty(\Omega)$, $\displaystyle \lim_{h\to0}\int_\Omega \rho(x,t_0+h)\varphi(x) \ dx=\int_\Omega \rho(x,t_0)\varphi(x) \ dx$. But moreover $\displaystyle \|\rho(t_0+h)\|_p\leq \max_{t\in[0,T]}\|\rho(t)\|_p$. So, by a classical argument in measure theory, we have that 
\begin{equation}\lim_{h\to0}\int_\Omega \rho(x,t_0+h)g(x) \ dx=\int_\Omega \rho(x,t_0)g(x) \ dx\end{equation}
for every $g\in L^q(\Omega)$. This implies \eqref{weakconvergenceLpLqtimetn1} and thus the thesis. 
\end{proof}


\section{Stability}

A fundamental question about transport equation is if the convergence to a certain limit of velocity fields and initial density implies the convergence (in a certain strong sense) of weak solutions. This fact holds, and it is known as \textit{stability theorem}. Before the statement of this theorem, we prove a result of consistence of the two notions of solution. 
\begin{theorem}\label{lemmaconsistency} Let $\rho\in L^\infty(0,T;L^p(\Omega))$ and $u\in L^1(0,T;L^q(\Omega))$ with $p\in (1,\infty]$. If $\rho$ is a renormalized solution, then $\rho$ is a weak solution. Moreover, if $\rho$ is a solution and $u\in L^1(0,T;W^{1,q}(\Omega))$, with $\nabla\cdot u=0$, then $\rho$ is a renormalized solution.
\end{theorem}
\begin{proof} We already know that if $u\in L^1(0,T;W^{1,q}(\Omega))$ and $\nabla\cdot u=0$, then $\rho\in L^\infty(0,T;L^p(\Omega))$ is a renormalized solution, thanks to theorem \ref{lemma82weakren}. We have to prove the other implication. Suppose that $\rho\in L^\infty(0,T;L^p(\Omega))$ is a renormalized solution to the problem. We want to prove that it is a weak solution. We can consider a sequence $\beta_k$ of admissible solution such that 
\begin{equation}|\beta_k(t)|\leq |t|, \qquad \beta_k(t)\to t\end{equation}
a $C^1$ approximation from below, bounded and with bounded derivatives. So we have 
\begin{equation}\label{weakrenformapprox48}-\int_0^T\bigg(\int_{\Omega} \beta_k(\rho) \frac{\partial \phi}{\partial t} \ dx\bigg) \ dt-\int_\Omega \beta_k(\rho_0) \phi(x,0) \ dx +\int_0^T\bigg(\int_{\Omega} \beta_k(\rho) u\cdot \nabla \phi \ dx\bigg) \ dt=0\end{equation}
We have now the bounds $|\beta_k(\rho)|\leq |\rho|$, and
\begin{equation}\int_0^T\bigg(\int_\Omega |\beta_k(\rho)|\left|\frac{\partial\phi}{\partial t}\right| \ dx\bigg) \ dt\leq \int_0^T\bigg(\int_\Omega |\rho|\left|\frac{\partial\phi}{\partial t}\right| \ dx\bigg) \ dt<\infty\end{equation}
since $\rho\in L^\infty(0,T;L^p(\Omega))$. Similarly, we have 
\begin{equation}\int_\Omega |\beta_k(\rho_0)||\phi(x,0)| \ dx\leq \int_\Omega |\rho_0||\phi(x,0)| \ dx<\infty\end{equation}
\begin{equation}\int_0^T\bigg(\int_\Omega |\beta_k(\rho)||u||\nabla \phi|\ dx \bigg) \ dt\leq \int_0^T\bigg(\int_\Omega |\rho||u||\nabla \phi|\ dx \bigg) \ dt<\infty\end{equation}
Since $\beta_k(t)\to t$ as $k\to\infty$ for every $t\in\R$, letting $k\to\infty$ in \eqref{weakrenformapprox48}, we have equation \ref{trasportodistirbuzionale} that is the weak formulation. 
\end{proof}
The following is the main theorem of the section. As typical, given a uniqueness result as the one in theorem \ref{unicitalemmarho0}, we expect a stability result
\begin{theorem}[Stability theorem] \label{teoremaglobalstability} 
Let $\Omega$ be a bounded domain, let $T>0$ and $p\in(1,\infty)$. Let $u^n\in L^1(0,T;C(\overline\Omega))$ be a sequence of velocity fields such that $u^n\big|_{\partial\Omega}\equiv0$. Suppose that exists 
\begin{equation}u\in L^1(0,T;W_0^{1,1}(\Omega))\cap L^1(0,T;C(\overline\Omega)), \quad \nabla\cdot u=0 \end{equation}
such that $u^n\to u$ in $L^1(0,T;L^1(\Omega))$. Let $\rho^n$ be a bounded sequence of measurable functions in $L^\infty(0,T;L^p(\Omega))$, that is $\displaystyle \sup_{n\in\N}\|\rho_n\|_{L^\infty(0,T;L^p(\Omega))}<\infty$, such that
\begin{equation}
-\int_0^T \bigg(\int_\Omega  \beta(\rho^n) \frac{\partial \phi}{\partial t}\ dx\bigg) \ dt-\int_\Omega \beta(\rho_n^0)\phi(x,0) \ dx +\int_0^T\bigg(\int_\Omega \beta(\rho^n) u^n\cdot \nabla \phi \ dx\bigg) \ dt=0
\end{equation}
for every $\beta\in \mathcal{A}$ and $\phi\in C_c^\infty(\Omega\times [0,T))$, and for some initial condition $\rho_n^0\in L^p(\Omega)$. Assume that 
\begin{equation}
\begin{cases}
\rho_n^0\to \rho^0 & \text{in $L^p(\Omega)$}\\
\beta(\rho_n^0)\to \beta(\rho^0) & \text{in $L^1(\Omega)$, $\forall \beta\in\mathcal{A}$}
\end{cases} \qquad n\to\infty
\end{equation}
Then $\rho^n$ converges, with respect to the norm of $L^\infty(0,T;L^p(\Omega))$, to some function $\rho\in L^\infty(0,T;L^p(\Omega))$ that is a renormalized solution of the transport equation with velocity field $u$ and initial density $\rho^0$, in $L^\infty(0,T;L^p(\Omega))$.
\end{theorem}
\begin{remark} We are supposing in this theorem the existence of a velocity field $u\in L^1(0,T;W_0^{1,1}(\Omega))$, with $\nabla \cdot u=0$, but asking merely the convergence $u^n\to u$ in $L^1(0,T;L^1(\Omega))$ and not in a more regular space. This is the crucial point of this approach. 
\end{remark}
\begin{remark} The aim of the theorem is to prove that $\rho_n\to \rho$ in $C([0,T];L^p(\Omega))$, where $\rho$ is a renormalized solution of the weak transport equation with velocity field $u$ and initial density $\rho^0$. If we know \textit{a priori} that $\rho_n\stackrel{*}{\rightharpoonup} \overline \rho$ in $L^\infty(0,T;L^\infty(\Omega))$ to some $\overline \rho\in L^\infty(0,T;L^p(\Omega))$, with $\overline \rho$ weak solution to the transport equation with field $u$ and initial density $\rho^0$, then by uniqueness theorem $\rho\equiv \overline \rho$, and so $\rho_n\to \overline \rho$ in $C([0,T];L^p(\Omega))$. This is the main application of this stability theorem (and the reason that lead to this statement). For example, it is useful to prove a stronger convergence in the article \cite{kimchoe1}.
\end{remark}
\begin{proof} We start with pointwise stability. Let $\beta$ an admissible function, and define $v_n:=\beta(\rho_n)$, where $\rho^n$ is renormalized solution to the transport equation with velocity filed $u^n$ and initial density $\rho_0^n$. Then, since $\beta$ is bounded, we have that $v_n\in L^\infty(0,T;L^\infty(\Omega))$. Moreover, observe that, by the hypothesis,
$$-\int_0^T \bigg(\int_\Omega  \beta(\rho^n) \frac{\partial \phi}{\partial t}\ dx\bigg) \ dt-\int_\Omega \beta(\rho_n^0)\phi(x,0) \ dx +\int_0^T\bigg(\int_\Omega \beta(\rho^n) u^n\cdot \nabla \phi \ dx\bigg) \ dt=0$$
and this can be rewritten as 
\begin{equation}-\int_0^T \bigg(\int_\Omega  v_n \frac{\partial \phi}{\partial t}\ dx\bigg) \ dt-\int_\Omega v_n^0\phi(x,0) \ dx +\int_0^T\bigg(\int_\Omega v_n \big(u^n\cdot \nabla \phi\big)\ dx\bigg) \ dt=0\end{equation}
where $\beta(\rho_n^0)=:v_n^0$. On the other hand, the function $\beta^2$ is admissible yet, and, as above, $w_n:=v_n^2\in L^\infty(0,T;L^\infty(\Omega))$. Moreover, as above, 
\begin{equation}-\int_0^T \bigg(\int_\Omega  w_n \frac{\partial \phi}{\partial t}\ dx\bigg) \ dt-\int_\Omega w_n^0\phi(x,0) \ dx +\int_0^T\bigg(\int_\Omega w_n \big(u^n\cdot \nabla \phi\big)\ dx\bigg) \ dt=0\end{equation}
and $w_n^0:=(v_n^0)^2$. Since the sequences are bounded in $L^\infty(0,T;L^\infty(\Omega))$, we have that exist $v,w\in L^\infty(0,T;L^\infty(\Omega))$ such that, up to extract a subsequence,
$$v_n\stackrel{*}{\rightharpoonup} v, \quad w_n\stackrel{*}{\rightharpoonup} w\quad \text{in $L^\infty(0,T;L^\infty(\Omega))$}$$
In particular 
\begin{equation}\|v\|_{L^\infty(0,T;L^\infty(\Omega))}\leq \liminf_{n\to\infty}\|v_n\|_{L^\infty(0,T;L^\infty(\Omega))}\leq C_\beta\end{equation}
where $\beta(s)\leq C_\beta$ for every $s\in \R$. 
Since $u^n\to u$ in $L^1(0,T;L^1(\Omega))$, we have that, considering for example the case of $v_n$ (that of $w_n$ is analogous), 
\begin{equation}\int_0^T\bigg(\int_\Omega v_n \frac{\partial \phi}{\partial t} \ dx \bigg) \ dt\to \int_0^T\bigg(\int_\Omega v \frac{\partial \phi}{\partial t} \ dx \bigg) \ dt, \qquad \int_\Omega v_n^0 \phi(0,x) \ dx \to \int_\Omega v^0 \phi(0,x) \ dx\end{equation}
since $\partial_t \phi\in L^1(0,T;L^1(\Omega))$ and $\phi(0,x)\in L^1(\Omega)$, and using the convergence of the densities in the hypothesis. Moreover
$$\bigg|\int_0^T\bigg(\int_\Omega v_n \big(u^n\cdot \nabla \phi\big)\ dx\bigg) \ dt-\int_0^T\bigg(\int_\Omega v \big(u\cdot \nabla \phi\big)\ dx\bigg) \ dt\bigg|\leq$$
$$\leq \bigg|\int_0^T\bigg(\int_\Omega v_n \big((u^n-u)\cdot \nabla \phi\big)\ dx\bigg) \ dt-\int_0^T\bigg(\int_\Omega (v-v_n) \big(u\cdot \nabla \phi\big)\ dx\bigg) \ dt\bigg|\leq$$
$$\leq \int_0^T\int_\Omega |v_n||u^n-u||\nabla \phi| \ dx \ dt+\bigg|\int_0^T\bigg(\int_\Omega (v-v_n) \big(u\cdot \nabla \phi\big)\ dx\bigg) \ dt\bigg|\leq$$
\begin{equation}\leq M\|v_n\|_{L^\infty(0,T;L^\infty(\Omega))}\|u^n-u\|_{L^1(0,T;L^1(\Omega))}+\bigg|\int_0^T\bigg(\int_\Omega (v-v_n) \big(u\cdot \nabla \phi\big)\ dx\bigg) \ dt\bigg|\end{equation}
where $M$ is such that $|\nabla\phi|\leq M$. Since $\|v_n\|_{L^\infty(0,T;L^\infty(\Omega))}$ is bounded and $u^n\to u$ in $L^1(0,T;L^1(\Omega))$, we only have to prove that also the other term vanishes. But 
\begin{equation}\int_0^T\int_\Omega |u||\nabla\phi|\ dx \ dt\leq M\|u\|_{L^1(0,T;L^1(\Omega))}<\infty\end{equation}
that is $u\cdot \nabla \phi\in L^1(0,T;L^1(\Omega))$, and since $v_n\stackrel{*}{\rightharpoonup} v$ in $L^\infty(0,T;L^\infty(\Omega))$, we have that also this term vanishes. So finally
\begin{equation}\label{traspeqforvren1}-\int_0^T\bigg(\int_\Omega v\frac{\partial \phi}{\partial t} \ dx\bigg)\ dt-\int_\Omega v^0(x) \ \phi(x,0) \ dx+\int_0^T\bigg(\int_\Omega v \ \big(u\cdot \nabla \phi\big) \ dx\bigg) \ dt=0\end{equation}
and, in the same way,
\begin{equation}-\int_0^T\bigg(\int_\Omega w\frac{\partial \phi}{\partial t} \ dx\bigg)\ dt-\int_\Omega w^0(x) \ \phi(x,0) \ dx+\int_0^T\bigg(\int_\Omega w \ \big(u\cdot \nabla \phi\big) \ dx\bigg) \ dt=0\end{equation}
Equation \eqref{traspeqforvren1} says that $v$ is a weak solution, with initial condition $v^0$; by the previous theorem it is a renormalized solution, since $u\in L^1(0,T;W^{1,q}(\Omega))$ and $\nabla\cdot u=0$.\\
\noindent Choosing $\alpha(t)=t^2$, approaching this function with admissible $\alpha_k(t)$ such that $\alpha_k(t)\leq t^2$ and $\alpha_k(t)\to t^2$ as $k\to\infty$, for every $t\in \R$. So we have that 
\begin{equation}-\int_0^T\bigg(\int_\Omega \alpha_k(v)\frac{\partial \phi}{\partial t} \ dx\bigg)\ dt-\int_\Omega \alpha_k(v^0(x)) \ \phi(x,0) \ dx+\int_0^T\bigg(\int_\Omega \alpha_k(v) \ \big(u\cdot \nabla \phi\big) \ dx\bigg) \ dt=0\end{equation}
implies, letting $k\to\infty$,
\begin{equation}-\int_0^T\bigg(\int_\Omega v^2\frac{\partial \phi}{\partial t} \ dx\bigg)\ dt-\int_\Omega (v^0)^2(x) \ \phi(x,0) \ dx+\int_0^T\bigg(\int_\Omega v^2 \ \big(u\cdot \nabla \phi\big) \ dx\bigg) \ dt=0\end{equation}
since $|v|^2\leq \|v\|_{L^\infty(0,T;L^\infty(\Omega))}^2\leq C_\beta^2$ and $v^0=\beta(\rho^0)\leq C_\beta$, so that the integrals are well-posed.\\
\noindent So, $v^2$ is a weak solution to the transport equation with initial condition $(v^0)^2$. But also $w$ is a weak solution to the same transport equation with initial condition $(v^0)^2$. By uniqueness theorem \ref{unicitalemmarho0}, since $v,w\in L^\infty(0,T;L^\infty(\Omega))$ and $u\in L^1(0,T;W_0^{1,1}(\Omega))\cap L^1(0,T;C(\overline\Omega))$, with $\nabla\cdot u=0$, we have $v^2\equiv w$.\\
\noindent This conclusion leads to the fact that $v_n^2\stackrel{*}{\rightharpoonup} v^2$ in $L^\infty(0,T;L^\infty(\Omega))$. Moreover, notice that 
\begin{equation}\|v_n-v\|_{L^2(0,T;L^2(\Omega))}^2=\langle v_n,v_n\rangle_{L^2(0,T;L^2(\Omega))}-2\langle v_n,v\rangle_{L^2(0,T;L^2(\Omega))}+\langle v,v\rangle_{L^2(0,T;L^2(\Omega))}\end{equation}
Observe that $\langle v_n,v\rangle_{L^2(0,T;L^2(\Omega))}\to \langle v,v\rangle_{L^2(0,T;L^2(\Omega))}$, since $v\in L^\infty(0,T;L^\infty(\Omega))\subset L^1(0,T;L^1(\Omega))$ and $v_n\stackrel{*}{\rightharpoonup} v$. Moreover, if we choose the function $\phi\equiv 1$ on $(0,T)\times \Omega$, that is in $L^1(0,T;L^1(\Omega))$, we have  
\begin{equation}\label{scelgophipernorm1a}\|v_n\|_{L^2(0,T;L^2(\Omega))}^2=\int_0^T\bigg(\int_\Omega |v_n|^2 \ dx\bigg)\ dt\equiv \int_0^T\bigg(\int_\Omega v_n^2\ \phi \ dx\bigg)\ dt\to\int_0^T\bigg(\int_\Omega v^2 \ \phi \ dx\bigg)\ dt=\int_0^T\bigg(\int_\Omega v^2 \ dx\bigg)\ dt\end{equation}
as $n\to\infty$, since $\phi\in L^1(0,T;L^1(\Omega))$ and $v_n^2\stackrel{*}{\rightharpoonup} v^2$ in $L^\infty(0,T;L^\infty(\Omega))$. This means that $v_n\to v$ in $L^2(0,T;L^2(\Omega))$.
\begin{remark} We can choose $\alpha(t)=|t|^p$, with $p\in (1,\infty)$, and obtain the same result in $L^p(0,T;L^p(\Omega))$. In fact, this choice implies that $|v_n|^p\stackrel{*}{\rightharpoonup} |v|^p$ in $L^\infty(0,T;L^\infty(\Omega))$. Moreover, $L^p(0,T;L^p(\Omega))$ is the dual of $L^q(0,T;L^q(\Omega))$, with $q$ and $q$ conjugate exponents. So, for every $\nu\in L^q(0,T;L^q(\Omega))$, we have $\langle v_n,\nu\rangle_{p,q} \to \langle v,\nu \rangle_{p,q}$, as $n\to\infty$, where $\langle \cdot,\cdot \rangle_{p,q}\equiv\langle \cdot ,\cdot \rangle_{L^p(0,T;L^p(\Omega)),L^q(0,T;L^q(\Omega))}$ is the dual pairing between $L^p(0,T;L^p(\Omega))$ and $L^q(0,T;L^q(\Omega))$. In fact 
$$\langle v_n,\nu\rangle_{p,q}=\int_0^T\bigg(\int_\Omega v_n\cdot \nu \ dx \bigg) \ dt\to \int_0^T\bigg(\int_\Omega v\cdot \nu \ dx \bigg) \ dt=\langle v,\nu \rangle_{p,q}$$
as $n\to\infty$, since $v_n\stackrel{*}{\rightharpoonup} v$ in $L^\infty(0,T;L^\infty(\Omega))$ and $\nu\in L^q(0,T;L^q(\Omega))\subseteq L^1(0,T;L^1(\Omega))$, being $q\geq1$. This means that $v_n\stackrel{*}{\rightharpoonup} v$ in $L^p(0,T;L^p(\Omega))$. Since, choosing $\phi\equiv 1$ as in \eqref{scelgophipernorm1a}, $|v_n|^p\stackrel{*}{\rightharpoonup} |v|^p$ in $L^\infty(0,T;L^\infty(\Omega))$ implies $\|v_n\|_{L^p(0,T;L^p(\Omega))}\to \|v\|_{L^p(0,T;L^p(\Omega))}$ as $n\to\infty$, it follows that $v_n\to v$ in $L^p(0,T;L^p(\Omega))$ in the strong sense.
\end{remark}
\noindent We now want to show that $v=\beta(\rho)$, for some $\rho\in L^\infty(0,T;L^p(\Omega))$, so that we have the convergence (in $L^p(0,T;L^p(\Omega))$) of $\beta(\rho_n)$ to $v$; this implies (since $v$ satisfies the weak transport equation) that $\rho$ is a renormalized solution (and so, if we assume also that $u\in L^1(0,T;L^q(\Omega))$, by theorem \ref{lemmaconsistency} a (unique) solution).\\
\noindent We know that $v_n=\beta(\rho_n)$ converges to $v\in L^2(0,T;L^2(\Omega))$ in $L^2(0,T;L^2(\Omega))\simeq L^2((0,T)\times \Omega)$. This implies that $v_n$ converges to $v$ in measure, that is $\beta(\rho_n)$ converges in measure to $v$. Since $|\Omega\times (0,T)|<\infty$ and, by the hypothesis, 
\begin{equation}\|\rho_n\|_{L^p(0,T;L^p(\Omega))}\leq C\|\rho_n\|_{L^\infty(0,T;L^p(\Omega))}\leq C\left(\sup_{n\in\N}\|\rho_n\|_{L^\infty(0,T;L^p(\Omega))}\right)<\infty\end{equation}
using propistion \ref{propositionconvmeas31}, we have that exists $\rho$, measurable function on $\Omega\times(0,T)$, such that $\rho_n\to \rho$ in measure. But, if $\beta\in C^1(\R)$ is an admissible function, we have, by theorem \ref{propositionconvmeas31}, that $v_n\equiv\beta(\rho_n)\to \beta(\rho)$ in measure. It follows that $v=\beta(\rho)$. In fact, we have 
\begin{equation}\|\beta(\rho)-v\|_{L^2(0,T;L^2(\Omega))}\leq \|\beta(\rho)-\beta(\rho_n)\|_{L^2(0,T;L^2(\Omega))}+\|\beta(\rho_n)-v\|_{L^2(0,T;L^2(\Omega))}\end{equation}
We know from above that $\|\beta(\rho_n)-v\|_{L^2(0,T;L^2(\Omega))}\to0$ as $n\to\infty$. On the other hand, $\beta(\rho_n)$ converges to $\beta(\rho)$ in measure and $|\beta(\rho_n)|\leq C_\beta$ implies that $\beta(\rho_n)$ has an integrable bound (uniform in $n$) in $L^2(0,T;L^2(\Omega))$. So, again by theorem \ref{propositionconvmeas31}, $\|\beta(\rho_n)-\beta(\rho)\|_{L^2(0,T;L^2(\Omega))}\to 0$ as $n\to\infty$. So, $\beta(\rho)=v\in L^2(0,T;L^2(\Omega))$.
\begin{remark} The same argument holds if $L^2(0,T;L^2(\Omega))$ is replaced by $L^p(0,T;L^p(\Omega))$. 
\end{remark}
\noindent So, the measurable function $\rho$ is a renormalized solution of the weak transport equation, since $v=\beta(\rho)$ is a solution. Now, using theorem \ref{corollariodatoinizialeconservato1a}, where only the measurability of $\rho$ is required, together with the fact that $u\in L^1(0,T;C(\overline\Omega))$, $u\big|_{\partial\Omega}\equiv0$, we deduce that $\|\rho(t_0)\|_p=\|\rho_0\|_p$ for almost every $t_0\in (0,T)$. So $\rho\in L^\infty(0,T;L^p(\Omega))$, and this implies that $\rho$ is a renormalized solution to the weak transport equation with initial density $\rho_0$, in the class $L^\infty(0,T;L^p(\Omega))$. That is, if $\beta$ is an admissible function, with $M>0$ such that $|\beta(s)|\leq M$ for every $s\in\R$, we have 
\begin{equation}-\int_0^T\bigg(\int_\Omega \beta(\rho) \ \partial_t \phi \ dx\bigg) \ dt -\int_\Omega \beta(\rho^0(x)) \phi(0,x)\ dx+\int_0^T\bigg(\int_\Omega \beta(\rho) \  (u\cdot \nabla\phi) \ dx\bigg) \ dt=0\end{equation}
Choosing $\phi\in C_c^\infty([0,T)\times\Omega)$ as in \eqref{reformofeq812813}, we have, for every $t_0\in[0,T]$ (eventually redefining the function out of a zero measure set)
\begin{equation}\int_\Omega \beta(\rho(t_0,x))\varphi(x) \ dx=\int_\Omega \beta(\rho^0(x)) \varphi(x)-\int_0^{t_0}\bigg(\int_\Omega \beta(\rho(x,t)) u(x,t)\cdot \nabla \varphi(x) \ dx\bigg) \ dt\end{equation}
Moreover, by the hypothesis, $\rho_n$ is renormalized solution to the transport equation with velocity field $u^n$ and initial density $\rho_n^0$. It follows that, if $t_0\in[0,T]$, after rearranging over a zero measure set, we have 
\begin{equation}\int_\Omega \beta(\rho_n(t_0,x))\varphi(x) \ dx=\int_\Omega \beta(\rho_n^0(x)) \varphi(x)-\int_0^{t_0}\bigg(\int_\Omega \beta(\rho_n(x,t)) u_n(x,t)\cdot \nabla \varphi(x) \ dx\bigg) \ dt\end{equation}
Let now $\{t_n\}_{n\in\N}\subseteq[0,T]$ be a sequence such that $t_n\to t_0\in [0,T]$, and consider that 
$$\int_\Omega \beta(\rho_n(t_n,x))\varphi(x) \ dx=\int_\Omega \beta(\rho_n^0(x)) \varphi(x)-\int_0^{t_n}\bigg(\int_\Omega \beta(\rho_n(x,t)) u_n(x,t)\cdot \nabla \varphi(x) \ dx\bigg) \ dt$$
We want to show that 
\begin{equation}\label{convergencebetaadmboundM1a}\lim_{n\to\infty}\int_\Omega \beta(\rho_n(t_n,x))\varphi(x) \ dx=\int_\Omega \beta(\rho(t_0,x))\varphi(x) \ dx\end{equation}
for every $\varphi\in C_c^\infty(\Omega)$. 
\begin{remark} [Proof of \eqref{convergencebetaadmboundM1a}] It is a calculation. In fact 
$$\int_\Omega \big(\beta(\rho_n(t_n,x))-\beta(\rho(t_0,x))\big) \varphi(x) \ dx=\int_\Omega \big(\beta(\rho_n^0(x))-\beta(\rho^0(x))\big) \varphi(x) \ dx-$$
$$-\bigg\{\int_0^{t_n}\bigg(\int_\Omega \beta(\rho_n(x,t)) u_n(x,t)\cdot \nabla \varphi(x) \ dx\bigg) \ dt-\int_0^{t_0}\bigg(\int_\Omega \beta(\rho(x,t)) u(x,t)\cdot \nabla \varphi(x) \ dx\bigg) \ dt\bigg\}=$$
$$=\int_\Omega \big(\beta(\rho_n^0(x))-\beta(\rho^0(x))\big) \varphi(x) \ dx-\int_{0}^{t_0}\bigg(\int_\Omega \beta(\rho_n) u_n\cdot \nabla \varphi \ dx -\int_\Omega \beta(\rho) u\cdot \nabla \varphi \ dx\bigg) \ dt \ +$$
\begin{equation}-\int_{t_0}^{t_n}\int_\Omega \beta(\rho_n) u_n\cdot \nabla \varphi \ dx \ dt\end{equation}
Observe, first of all, that 
\begin{equation}\bigg|\int_\Omega \big(\beta(\rho_n^0(x))-\beta(\rho^0(x))\big) \varphi(x) \ dx\bigg|\leq \|\varphi\|_\infty \|\beta(\rho_n^0)-\beta(\rho^0)\|_1\to 0\end{equation}
as $n\to\infty$. Furthermore
$$\bigg|\int_{0}^{t_0}\bigg(\int_\Omega \beta(\rho_n) u_n\cdot \nabla \varphi \ dx -\int_\Omega \beta(\rho) u\cdot \nabla \varphi \ dx\bigg) \ dt\bigg|=\bigg|\int_0^{t_0}\bigg(\int_\Omega \big(\beta(\rho_n)u_n-\beta(\rho) u\big) \cdot \nabla \varphi \ dx\bigg) \ dt\bigg|=$$
$$=\bigg|\int_0^{t_0}\bigg(\int_\Omega \beta(\rho_n)\big(u_n- u\big) \cdot \nabla \varphi \ dx\bigg) \ dt+\int_0^{t_0}\bigg(\int_\Omega \big(\beta(\rho_n)-\beta(\rho)\big) u \cdot \nabla \varphi \ dx\bigg) \ dt\bigg|\leq$$
\begin{equation}\leq M\|\nabla\varphi\|_\infty\int_0^T\int_\Omega |u_n-u| \ dx \ dt+\bigg|\int_0^{t_0}\bigg(\int_\Omega \big(\beta(\rho_n)-\beta(\rho)\big) u \cdot \nabla \varphi \ dx\bigg) \ dt\bigg|\to 0\end{equation}
as $n\to\infty$, since $u^n\to u$ in $L^1(0,T;L^1(\Omega))$ and $\beta(\rho_n)\stackrel{*}{\rightharpoonup}v=\beta(\rho)$ in $L^\infty(0,T;L^\infty(\Omega))$ and $\chi_{(0,t_0)}u\cdot \nabla \varphi\in L^1(0,T;L^1(\Omega))$. Moreover, we have 
$$\bigg|\int_{t_0}^{t_n}\int_\Omega \beta(\rho_n) u_n\cdot \nabla \varphi \ dx \ dt\bigg|\leq\bigg|\int_{t_0}^{t_n}\bigg(\int_\Omega \beta(\rho_n) u_n\cdot \nabla \varphi \ dx -\int_\Omega \beta(\rho) u \cdot \nabla \varphi \ dx\bigg) \ dt\bigg|+$$
$$+\bigg|\int_{t_0}^{t_n}\bigg(\int_\Omega \beta(\rho) u\cdot \nabla \varphi\bigg) \ dt\bigg|\leq \bigg|\int_0^T\chi_{(t_0,t_n)}(t)\bigg(\int_\Omega \beta(\rho_n) u_n\cdot \nabla \varphi \ dx -\int_\Omega \beta(\rho) u \cdot \nabla \varphi \ dx\bigg) \ dt\bigg|+$$
$$+M\|\nabla\varphi\|_\infty\int_{t_0}^{t_n}\|u\|_1 \ dt\leq \bigg|\int_0^T\chi_{(t_0,t_n)}(t)\bigg(\int_\Omega \beta(\rho_n)(u_n-u)\cdot \nabla \varphi \ dx +\int_\Omega (\beta(\rho_n)-\beta(\rho))u\cdot \nabla \varphi \ dx\bigg) \ dt\bigg|+$$
\begin{equation}+M\|\nabla\varphi\|_\infty\int_{t_0}^{t_n}\|u\|_1 \ dt\leq M\|\nabla\varphi\|_\infty\|u_n-u\|_{L^1(0,T;L^1(\Omega))}+3M\|\nabla\varphi\|_\infty\int_{t_0}^{t_n}\|u\|_1 \ dt\to0\end{equation}
as $n\to\infty$, since $t_n\to t_0$, $u^n\to u$ in $L^1(0,T;L^1(\Omega)$ as $n\to\infty$. So we have proved \eqref{convergencebetaadmboundM1a}. 
\end{remark}
\noindent Starting from \eqref{convergencebetaadmboundM1a}, we want to show that also 
\begin{equation}\label{weakconvergencewoutbeta1a}\lim_{n\to\infty}\int_\Omega \rho_n(t_n,x) \varphi(x) \ dx =\int_\Omega \rho(t_0,x) \varphi(x) \ dx\end{equation}
for every $\varphi\in C_c^\infty(\Omega)$ and $t_n\to t_0$. 
\begin{remark}\label{osservazioneconvdebpq1a} If \eqref{weakconvergencewoutbeta1a} holds, then it is true for every $\varphi\in L^q(\Omega)$. Moreover, we have that 
\begin{equation}\|\rho_n(t_n)\|_p=\|\rho_0^n\|_p\stackrel{n\to\infty}{\longrightarrow} \|\rho_0\|_p=\|\rho_0(t_0)\|_p\end{equation}
thanks to the convergence of $\rho_0^n\to \rho_0$ in $L^p(\Omega)$ by hypothesis and using corollary \ref{corollariodatoinizialeconservato1a}.
So, it follows that $\rho_n(t_n)\to \rho(t_0)$ in $L^p(\Omega)$. From theorem \ref{lemmaan194} it follows that $\rho_n\to \rho$ in $C([0,T];L^p(\Omega))$, as $n\to\infty$, that is the thesis. 
\end{remark}
\begin{remark}[Proof of \eqref{weakconvergencewoutbeta1a}] So we have to prove \eqref{weakconvergencewoutbeta1a}. Given $M\in (0,\infty)$, consider the function
\begin{equation}\beta_M(s):=\begin{cases} s & |s|\leq M\\
M & s>M\\
-M & s<-M
\end{cases}\end{equation}
We have to fix this $M$ in a precise way. Let $t_n\to t_0\in [0,T]$ and a consider the sequence $\{\rho_n(t_n),\rho(t_0)\}_{n\in \N}\subseteq L^p(\Omega)$. Moreover, this sequence is bounded in $L^p(\Omega)$,
since $\displaystyle \|\rho_n(t_n)\|_p\equiv \|\rho_0^n\|_p\leq C$ by the convergence in the hypothesis. So, using lemma \ref{teoremaunifintgfn1a}, we have that for every $\varepsilon>0$ exists $M_\varepsilon>0$ such that 
\begin{equation}\label{unifintofthisseqdens1a}\int_{\{x\in\Omega: \ |\rho(t_0,x)|>M_\varepsilon\}}|\rho(t_0,x)| \ dx<\varepsilon, \quad \int_{\{x\in\Omega: \ |\rho_n(t_n,x)|>M_\varepsilon\}}|\rho_n(t_n,x)| \ dx<\varepsilon, \qquad \forall \ n\in\N\end{equation}
Notice that \eqref{unifintofthisseqdens1a} implies that 
$$M_\varepsilon|\{x\in\Omega: \ |\rho(t_0,x)|>M_\varepsilon\}|<\varepsilon,\quad M_\varepsilon |\{x\in\Omega: \ |\rho_n(t_n,x)|>M_\varepsilon\}|<\varepsilon, \qquad \forall  \ n\in\N$$
that will be useful in a moment. Fix $\varepsilon>0$ and choose $M_\varepsilon>0$ as above. Then we can consider $\beta_{M_\varepsilon}$. Moreover, let $\beta_{M_\varepsilon}^k$ an admissible functions that coincides with $\beta_{M_\varepsilon}$ except two neighbourhoods, of $s=\pm M_\varepsilon$, and such that 
\begin{equation}|\beta_{M_\varepsilon}^{k}(s)|\leq |\beta_{M_\varepsilon}(s)|\leq M_\varepsilon, \qquad \sup_{s\in\R}|\beta_{M_\varepsilon}^k(s)-\beta_{M_\varepsilon}(s)|<\frac1k\end{equation} 
We can choose $k_\varepsilon\in\N$ such that $\displaystyle \frac{|\Omega|\|\varphi\|_\infty}{k_\varepsilon}<\varepsilon$. So, we can write 
\begin{equation}\label{scritturaspezzataintconvdebp1a}\int_\Omega \rho_n(t_n,x)\varphi(x) \ dx=\int_\Omega \beta_{M_\varepsilon}(\rho_n(t_n,x)) \varphi(x) \ dx +\int_\Omega \{\rho_n(t_n,x)-\beta_{M_\varepsilon}(\rho_n(t_n,x))\} \varphi(x) \ dx\end{equation}
We, at first, focus our attention to the second addend. We have 
$$\bigg|\int_\Omega \{\rho_n(t_n,x)-\beta_{M_\varepsilon}(\rho_n(t_n,x))\} \varphi(x) \ dx\bigg|=\bigg|\int_{\{x\in\Omega: \ |\rho_n(t_n,x)|\leq M_\varepsilon\}} \{\rho_n(t_n,x)-\beta_{M_\varepsilon}(\rho_n(t_n,x))\} \varphi(x) \ dx+$$
$$+\int_{\{x\in\Omega: \ |\rho_n(t_n,x)|>M_\varepsilon\}} \{\rho_n(t_n,x)-\beta_{M_\varepsilon}(\rho_n(t_n,x))\} \varphi(x) \ dx\bigg|=\bigg|\int_{\{x\in\Omega: \ |\rho_n(t_n,x)|>M_\varepsilon\}} \{\rho_n(t_n,x)-M_\varepsilon\} \varphi(x) \ dx\bigg|\leq$$
\begin{equation}\leq \|\varphi\|_\infty\bigg(\int_{\{x\in\Omega: \ |\rho_n(t_n,x)|>M_\varepsilon\}}|\rho_n(t_n,x)| \ dx +M_\varepsilon |\{x\in\Omega: \ |\rho_n(t_n,x)|>M_\varepsilon\}|\bigg)<2\varepsilon\|\varphi\|_\infty\end{equation}
If in equation \eqref{scritturaspezzataintconvdebp1a} we subtract the term $\displaystyle \int_\Omega \rho(t_0,x)\varphi(x) \ dx$, we have also to consider 
$$\bigg|\int_\Omega \beta_{M_\varepsilon}(\rho_n(t_n,x)) \varphi(x) \ dx- \int_\Omega \rho(t_0,x)\varphi(x) \ dx\bigg|\leq$$
\begin{equation}\leq\bigg|\int_\Omega \big(\beta_{M_\varepsilon}(\rho_n(t_n,x))-\beta_{M_\varepsilon}(\rho(t_0,x))\big) \varphi(x) \ dx\bigg|+\bigg| \int_\Omega \big(\beta_{M_\varepsilon}(\rho(t_0,x))-\rho(t_0,x)\big)\varphi(x) \ dx\bigg|\end{equation}
We deal at first with the second addend. Following the steps above, we have again 
\begin{equation}\bigg|\int_\Omega \big(\beta_{M_\varepsilon}(\rho(t_0,x))-\rho(t_0,x)\big)\varphi(x) \ dx\bigg|\leq 2\varepsilon\|\varphi\|_\infty\end{equation}
The other term can be written as 
$$\bigg|\int_\Omega \big(\beta_{M_\varepsilon}(\rho_n(t_n,x))-\beta_{M_\varepsilon}(\rho(t_0,x))\big) \varphi(x) \ dx\bigg|\leq$$
$$\leq\|\varphi\|_\infty\int_\Omega |\beta_{M_\varepsilon}(\rho_n(t_n,x))-\beta_{M_\varepsilon}^{k_\varepsilon}(\rho_n(t_n,x))| \ dx+\bigg|\int_\Omega \big(\beta_{M_\varepsilon}^{k_\varepsilon}(\rho_n(t_n,x))-\beta_{M_\varepsilon}^{k_\varepsilon}(\rho(t_0,x))\big) \varphi(x) \ dx\bigg| \ +$$
$$+ \ \|\varphi\|_\infty\int_\Omega |\beta_{M_\varepsilon}^{k_\varepsilon}(\rho(t_0,x))-\beta_{M_\varepsilon}(\rho(t_0,x))| \ dx\leq $$
\begin{equation}\leq \frac{2\|\varphi\|_\infty |\Omega|}{ k_\varepsilon}+\bigg|\int_\Omega \big(\beta_{M_\varepsilon}^{k_\varepsilon}(\rho_n(t_n,x))-\beta_{M_\varepsilon}^{k_\varepsilon}(\rho(t_0,x))\big) \varphi(x) \ dx\bigg|\end{equation}
We have that, for every admissible function, \eqref{convergencebetaadmboundM1a} holds, and so there exists $N=N(\beta_{M_\varepsilon}^{k_\varepsilon})\equiv N_\varepsilon$ such that, for every $n\geq N_\varepsilon$,
\begin{equation}\bigg|\int_\Omega \rho_n(t_n,x)\varphi(x) \ dx-\int_\Omega \rho(t_0,x)\varphi(x) \ dx\bigg|\leq 4\varepsilon\|\varphi\|_\infty+ \frac{2\|\varphi\|_\infty |\Omega|}{ k_\varepsilon} +\varepsilon\leq 4\varepsilon \|\varphi\|_\infty+3\varepsilon\end{equation}
that is 
\begin{equation}\lim_{n\to\infty}\int_\Omega \rho_n(t_n,x)\varphi(x) \ dx=\int_\Omega \rho(t_0,x)\varphi(x) \ dx\end{equation}
\end{remark}
This concludes the proof.
\end{proof}


%
%
%
%
%


%
%
%



\renewcommand{\refname}{\spacedlowsmallcaps{References}} 

\bibliographystyle{unsrt}

\bibliography{sample.bib} 


\end{document}

%% file: structure.tex
\usepackage[
nochapters, 
beramono, 
pdfspacing, 
dottedtoc 
]{classicthesis} 

\usepackage{arsclassica} 

\usepackage[T1]{fontenc} 

\usepackage[utf8]{inputenc} 

\usepackage[a4paper,top=1.5cm,bottom=2cm,left=2.5cm,right=2.5cm]{geometry}

\usepackage{graphicx} 
\graphicspath{{Figures/}} 

\usepackage{enumitem} 

\usepackage{lipsum} 

\usepackage{subfig} 

\usepackage{amsmath,amssymb,amsthm} 

\usepackage{varioref} 

\usepackage{hyperref}

\usepackage[bottom]{footmisc}

\theoremstyle{definition} 
\newtheorem{definition}{Definition}

\theoremstyle{plain} 
\newtheorem{theorem}{Theorem}

\theoremstyle{plain} 
\newtheorem{lemma}{Lemma}

\theoremstyle{remark} 
\newtheorem{remark}{Remark}

\hypersetup{
colorlinks=true, breaklinks=true, bookmarks=true,bookmarksnumbered,
urlcolor=webbrown, linkcolor=RoyalBlue, citecolor=webgreen, 
pdftitle={}, 
pdfauthor={\textcopyright}, 
pdfsubject={}, 
pdfkeywords={}, 
pdfcreator={pdfLaTeX}, 
pdfproducer={LaTeX with hyperref and ClassicThesis} 
}